\documentclass{amsart}
\makeatletter

\usepackage{amsmath,amssymb} 
\usepackage{amsthm}
\usepackage[all]{xy}
\usepackage{extarrows}
\usepackage{graphicx}
\usepackage{bm}

\newtheorem{thm}{Theorem}[section]
\newtheorem{lem}[thm]{Lemma}
\newtheorem{cor}[thm]{Corollary}
\newtheorem{prop}[thm]{Proposition}  

\theoremstyle{remark}
\newtheorem{ack}{Acknowledgments}        
\newtheorem{nota}{Notation}    

\theoremstyle{definition}
\newtheorem{defn}[thm]{Definition}
\newtheorem{rem}[thm]{Remark} 

\newtheorem{eg}[thm]{Examples}

\numberwithin{equation}{section}

\makeatother

\title{Id\`elic class field theory for 3-manifolds}
\date{}
\author{Hirofumi Niibo}

\begin{document}
\bibliographystyle{amsalpha+}
\maketitle


\begin{abstract}

Following the analogies between 3-dimensional topology and number theory, we study an id\`elic form of class field theory for 3-manifolds. 
For a certain set $\mathcal{K}$ of knots in a 3-manifold $M$, 
we first present a local theory for each knot in $\mathcal{K}$, which is analogous to local class field theory,  and then, 
getteing together over all knots in $\mathcal{K}$, we give an analogue of id\`elic  global class field theory for an integral homology sphere $M$.
\end{abstract}



\section{Introduction}

\noindent

\indent The analogies between knots and primes were firstly pointed out by B. Mazur (\cite{mazur}) in 1960's and, after a long silence, 
M. Kapranov and A. Reznikov took up the analogies between 3-manifolds and number rings again (\cite{kapranov}, \cite{reznikov1}, \cite{reznikov2}), and M. Morishita, whose work started independently, 
investigated the subject systematically (\cite{mor2}, \cite{mor3}, \cite{mor1}). 
This new area of mathmatics is now called {\it arithmetic topology}. \\

\par Here is a part of basic analogies: For a number field $k$, ${\mathcal O}_{k}$ stands for the ring of integers of $k$.

$$ \begin{array}{ccc}
 \text{ 3-manifold }M &  \longleftrightarrow & \text{number ring Spec}({\mathcal O}_{k}) \\
 \text{ knot }K \text{ in }M & \longleftrightarrow & \text{ prime }\mathfrak{p} \text{ in Spec}({\mathcal O}_{k}) \\
 \text{ (ramified) covering }N\rightarrow M  &  \longleftrightarrow & \text{ (ramified) extension } K/k\\
 \text{ 1st homology group } H_1(M;\mathbb{Z})  & \longleftrightarrow  &  \text{ ideal class group } H_k
\end{array}\leqno{(1.1)}
$$
In particular, we have the following analogy between the Hurewicz isomorphism and unramified class field theory (E. Artin's isomorphism): 

$$
\begin{array}{ccc}
 H_1(M;\mathbb{Z}) \cong \mathrm{Gal}(M^{\mathrm{ab}}/M) & \longleftrightarrow &  H_k \cong \mathrm{Gal}(\tilde{k}^{\mathrm{ab}}/k)
\end{array}\leqno{(1.2)}
$$
Here $M^{\mathrm{ab}}$(resp. $\tilde{k}^{\mathrm{ab}}$) denotes the maximal Abelian covering of $M$ 
(resp. the maximal unramified Abelian extension of $k$). \\

\par The purpose of this paper is, following the spirit of arithmetic topology, 
to pursue an id\`ele theoretic form of class field theory for 3-manifolds 
and so extend the analogy (1.2) for ramified coverings/extensions. \\

\par For this, we first develop a local theory for each knot in a 3-manifold,
which is analogous to local class field theory, based on the following analogies. \\

$$
\begin{array}{ccc}
 \text{tubular neighborhood of }K & &  \mathfrak{p}\text{-adic integers}\\
 V_K & \longleftrightarrow &  \mathrm{Spec}(\mathcal{O}_{\mathfrak{p}})\\
 \text{boundary of }V_K & & \mathfrak{p}\text{-adic field} \\
 \partial V_K \cong V_K\setminus K & \longleftrightarrow & \mathrm{Spec}(k_{\mathfrak{p}}) = \rm{Spec}(\mathcal{O}_{\mathfrak{p}}) \setminus \rm{Spec}(\mathcal{O}_{\mathfrak{p}}/\mathfrak{p})
\end{array}\leqno{(1.3)}
$$

Then a topological analogue of the local reciprocity homomorphism is simply given by the Hurewicz homomorphism:
$$\rho_K:H_1(\partial V_K ; \mathbb{Z}) \rightarrow \mathrm{Gal}(\partial V_K^{\mathrm{ab}}/\partial V_K).$$

\par For a certain given set $\mathcal{K}$ of knots in a 3-manifold $M$ (cf. Section 5), 
we introduce the {\it id\`ele group} $I_{(M;\mathcal{K})}$ as a restricted product of $H_1(\partial V_K ;\mathbb{Z})$ over all $K$ in $\mathcal{K}$, 
and getting $\rho_K$'s together over all $K$ in $\mathcal{K}$, define the homomorphism

$$
\varphi_{(M;\mathcal{K})}:I_{(M;\mathcal{K})} \rightarrow \mathrm{Gal}(M;\mathcal{K})^{\mathrm{ab}}:= \varprojlim_{L} \mathrm{Gal}(X_{L}^{\mathrm{ab}} / X_L)
$$
where $L$ runs over all finite subsets of $\mathcal{K},\; X_L := M \setminus L$ and $X_{L}^{\mathrm{ab}}$ is the maximal Abelian covering of $X_L$. 
The homomorphism $\phi_{(M;\mathcal{K})}$ factors through the {\it id\`ele class group} 
$C_{(M;\mathcal{K})}:= I_{(M;\mathcal{K})}/P_{(M;\mathcal{K})}$ with the {\it principal id\`ele group} $P_{(M;\mathcal{K})}$, 
and hence we obtain an analogue of the global reciprocity homomorphism
$$ \rho_{(M;\mathcal{K})}: C_{(M;\mathcal{K})} \rightarrow \mathrm{Gal}(M;\mathcal{K})^{\mathrm{ab}}.$$
Then our main result (Theorem 5.9 below) is stated as follows. 
Suppose that $M$ is an integral homology sphere. 
For a finite Abelian covering $h: N \rightarrow M$ branched over a finite subset of $\mathcal{K}$, 
the global reciprocity homomorphism $\rho_{(M;\mathcal{K})}$ induces an isomorphism
$$
\rho_{N/M}:C_{(M;\mathcal{K})}/h_*(C_{(N,h^{-1}(\mathcal{K}))}) \cong \mathrm{Gal}(N/M).
$$
This result may be regarded as an analogue of the fundamental theorem in global class field theory for number fields (\cite{kato}). \\

We note that id\`{e}lic class field theory for 3-manifolds was firstly studied by A. Sikora (\cite{sikora1}, \cite{sikora2}). 
Our approach is different from his and elementary. \\

Here is the contents of this paper. 
In Section 2 we review the class field theory for algebraic fields. 
In Section 3 we give a description of Hilbert theory for 3-manifolds. 
In Section 4 we give the local class field theory for tori, 
and section 5 we present the global class field theory over on integral homology 3-sphere.


\begin{nota}
For a connected topological space $X$ (resp. a field $k$),
we denote by $X^{\rm{ab}}$ (resp. $k^{\rm{ab}}$) the maximal Abelian covering of $X$ (resp. the maximal Abelian extension of $k$). 
We denote by $\mathrm{Gal}(Y/X)$ (resp. $\mathrm{Gal}(F/k)$) for the Galois group of a Galois covering $h:Y \rightarrow X$ (resp. a Galois extension $F/k$). 
We write $\pi_1(X)$ for the fundamental group of $X$ omitting a base point and write $H_n(X)$ simply for the homology group with coefficients in $\mathbb{Z}$.

\end{nota}


\begin{ack}
I would like to thank my supervisor Professor Masanori Morishita for his advice and encouragements. 
I would also like to express my deep gratitude to my family for their support.
\end{ack}


\section{Review of class field theory in number theory}

\noindent
\indent
In this section, we review local and global class field theory in number theory whose topological analogies will be studied in the sections 4 and 5. 
We consult \cite{kato} and \cite{Neukirch} as basic references for this section. \\

 Let $k$ be a number field of finite degree over the rational number field $\mathbb{Q}$.
We denote by $\mathcal{O}_k$ the ring of integers of $k$. 
A prime $\mathfrak{p}$ of $k$ is a class of equivalent valuations of $k$. 
The finite primes belong to the maximal ideals of $\mathcal{O}_k$. 
The infinite primes fall into two classes, the real and complex ones, where the real primes correspond to the embeddings $k \hookrightarrow \mathbb{R}$ 
and the complex primes correspond to the pairs of conjugate non-real embeddings $k \hookrightarrow \mathbb{C}$. 
For a finite prime $\mathfrak{p}$, let  $v_\mathfrak{p}$ be the corresponding valuation of $k$,
and set $\vert a \vert_\mathfrak{p} = (N\mathfrak{p})^{-v_\mathfrak{p}(a)}$ for $a \in k$ where $N\mathfrak{p}=\#(\mathcal{O}_k/\mathfrak{p})$.
For a real prime $\mathfrak{p}$ with corresponding embedding $\iota : k \hookrightarrow \mathbb{R}$, 
set $\vert a \vert_\mathfrak{p} = \vert \iota(a) \vert$ for $a \in k$, 
and for a complex prime $\mathfrak{p}$ with corresponding embedding $\iota:k \hookrightarrow \mathbb{C}$, 
set $\vert a \vert_\mathfrak{p} =\vert \iota(a) \vert^2$ for $a \in k$.\\

\par Let $k_{\frak{p}}$ be the local field obtained as the completion of a number field $k$ with respect to the metric $\vert \cdot \vert_{\mathfrak{p}}$. 
Suppose that  $\mathfrak{p}$ is a finite prime of $k$. 
Then $k_\mathfrak{p}$ is non-archimedian local field, a finite extension of the $p$-adic field $\mathbb{Q}_p$ for a prime number $p$. 
Let $v_{\frak{p}} : k_{\frak{p}}^\times \rightarrow \mathbb{Z}$ be the discrete valuation normalized by $v_{\frak{p}}(k_{\frak{p}}^\times)=\mathbb{Z}$. 
We let $\mathcal{O}_{\frak{p}} :=\{\, a \in k_{\frak{p}} \mid v_{\mathfrak{p}}(a) \geq 0 \,\}$ be the valuation ring
and $\mathfrak{p} = \{\, a \in k \, | \,  v_{\mathfrak{p}}(a) > 0 \,\}$ be the unique maximal ideal of $\mathcal{O}_{\frak{p}}$, 
and let $\mathbb{F}_{\frak{p}}$ be the residue field $\mathcal{O}_{\frak{p}}/\mathfrak{p}$, a finite extension of $\mathbb{F}_p = \mathbb{Z}/p\mathbb{Z}$. 
We denote by $U_{\frak{p}}$ the unit group ${\mathcal O}_{\frak{p}}^{\times}$. 
We note that $U_{\frak{p}} = \mathrm{Ker}(v_{\mathfrak{p}})$ and so we have the following split exact sequence 
$$ 0 \longrightarrow U_{\mathfrak{p}} \longrightarrow k_{\mathfrak{p}}^\times \xlongrightarrow{v_\mathfrak{p}} \mathbb{Z} \longrightarrow 0 . \leqno{(2.1)}$$
When $\mathfrak{p}$ is an infinite prime, we let $\mathcal{O}_\mathfrak{p}=k_\mathfrak{p}$ and $U_\mathfrak{p}=k_{\mathfrak{p}}^\times$ by convention.

\par Let $k_{\frak{p}}^{\mathrm{ab}}$ be the maximal Abelian extension of $k_{\mathfrak{p}}$. 
When $k_{\frak{p}}$ is non-archimedian, we denote by $k_{\frak{p}}^{\mathrm{ur}}$ the maximal unramified extension of $k_{\frak{p}}$. 
Note that the Galois group $\mathrm{Gal}(k_{\mathfrak{p}}^{\mathrm{ur}}/k_{\mathfrak{p}})$ is identified with $\mathrm{Gal}(\bar{\mathbb{F}}_{\mathfrak{p}}/\mathbb{F}_{\frak{p}}) \cong \hat{\mathbb{Z}}$, 
where $\hat{\mathbb{Z}}$ denotes the profinite completion of $\mathbb{Z}$. 
A main part of local class field theory for the local field $k_{\mathfrak{p}}$ is stated as follows.

	\begin{thm}[Local class field theory]\label{local CFT} 
There is a canonical homomorphism, called the local reciprocity homomorphism, 
$$\rho_{k_\mathfrak{p}}:k_\mathfrak{p}^{\times} \rightarrow \mathrm{Gal}(k_\mathfrak{p}^{\mathrm{ab}}/k_\mathfrak{p})$$ 
which satisfies the following properties:

		$(1)$ For any finite Abelian extension $F/k_\mathfrak{p}$, $\rho_{k_\mathfrak{p}}$ induces the isomorphism 
		 $$
                 	\rho_{F/k_\mathfrak{p}}:k_\mathfrak{p}^{\times}/N_{F/k_\mathfrak{p}}(F^{\times})\cong \mathrm{Gal}(F/k_\mathfrak{p}) 
                 $$
                 where $N_{F/k_\mathfrak{p}}$ denotes the norm map for $F/k_\mathfrak{p}$.

		$(2)$ When $k_\mathfrak{\mathfrak{p}}$ is non-archimedean, we have the following commutative diagrams with exact horizontal sequences: 
                \[\xymatrix{
			0 \ar[r] & U_{\mathfrak{p}} \ar[r] \ar[d]^{ \rho_{ k_\mathfrak{p} } \vert_{U_\mathfrak{p} }} & k_\mathfrak{p}^\times \ar[d]^{\rho_{k_\mathfrak{p}}} \ar[r]^{v_{\mathfrak{p}}} \ar@{}[dr] & \mathbb{Z} \ar[d]^{} \ar[r] & 0 \\
			0 \ar[r] & \mathrm{Gal}(k_{\mathfrak{p}}^{\mathrm{ab}}/k_{\mathfrak{p}}^{\mathrm{ur}}) \ar[r] & \mathrm{Gal}(k_{\mathfrak{p}}^{\mathrm{ab}}/ k_\mathfrak{p}) \ar[r] & \mathrm{Gal}(\bar{\mathbb{F}}_{\mathfrak{p}}/\mathbb{F}_{\frak{p}}) \ar[r] & 0. \\
			}\]
	
	\end{thm}  

 \begin{cor} \label{cor of local CFT}
 There is the one to one correspondence between finite unramified extensions of $k_{\frak{p}}$ and 
 open subgroups of finite index of $k_{\frak{p}}^\times$ containing $U_{\mathfrak{p}}$.
 \end{cor}
$\\$
\par Now, let $k$ be a number field. We define the id\`ele group $I_k$ of $k$ by the following restricted product of $k_{\mathfrak{p}}^\times$'s with respect to $U_{\mathfrak{p}}$'s  over all primes ${\mathfrak{p}}$ of $k$:
$$I_k:= 
\left \{\, (a_\mathfrak{p})_{\mathfrak{p}} \in \displaystyle \prod_{\mathfrak{p}:\text{ prime}} k_\mathfrak{p}^{\times} \mid
 v_{\mathfrak{p}}(a_{\mathfrak{p}})=0 \text{ for almost all finite prime } \mathfrak{p} \,\right\}.$$

Since, for $a \in k^\times$, we have $v_{\mathfrak{p}}(a)=0$ for almost all finite prime $\mathfrak{p}$, 
$k^\times$ is embedded $I_k$ diagonary. 
We let $P_k$ be the image of $k^\times$ in $I_k$ and call it the group of principal id\`eles. 
We then define the id\`ele class group of $k$ by
$$C_k := I_k/k^\times .$$
We recall that the homomorphism
$$ \varphi:I_k \rightarrow \bigoplus_{\mathfrak{p}\text{ prime}} \mathbb{Z} \ ; \  
(a_\mathfrak{p})_\mathfrak{p} \mapsto \prod_{\mathfrak{p}}\mathfrak{p}^{v_{\mathfrak{p}}(a_{\mathfrak{p}})} $$
induces the isomorphism
$$ I_k/U\cdot P_{k} \cong H_k \leqno{(2.3)}$$
where $U=\mathrm{Ker}(\varphi)=\prod_{\mathfrak{p}}U_\mathfrak{p}$, and $H_k$ denotes the ideal class group of $k$.

Let $k^{\mathrm{ab}}$ be the maximal Abelian extension of $k$. Here is a global class field theory for $k$ (cf \cite{Neukirch}).
A main part of global class field theory is summarized as follows.

\begin{thm}[Global class field theory]\label{global CFT}
There is a canonical homomorphism, called the global reciprocity map,
$$\rho_{k}:C_k \rightarrow \mathrm{Gal}(k^{\mathrm{ab}}/k)$$ which has tha following properties: 

$(1)$ For any finite Abelian extension $F/k$,
$\rho_k$ induces the isomorphism.
$$C_k/N_{F/k}(C_F)\cong \mathrm{Gal}(F/k)$$
where $N_{F/k}$ denotes the norm map on the id\`ele groups.

$(2)$ For a prime $\mathfrak{p}$ of $k$, we have the following commutative diagram
\[\xymatrix{
k_\mathfrak{p}^{\times} \ar[d]_{\iota_\mathfrak{p}} \ar[r]^(0.3){\rho_{k_\mathfrak{p}}} \ar@{}[dr]|\circlearrowleft & \mathrm{Gal}(k_\mathfrak{p}^{\mathrm{ab}}/k_{\mathfrak{p}}) \ar[d]^{} \\
C_k \ar[r]_(0.3){\rho_k} & \mathrm{Gal}(k^{\mathrm{ab}}/k) \\
}\]
where $\iota_{\mathfrak{p}}$ is the map induced by the natural inclusion $k_\mathfrak{p}^\times \rightarrow I_k$.       
\end{thm}

By class field theory, we obtain the following proposition. 

\begin{prop}\label{cor of global CFT}
For a finite Abelian extension $F/k$, let $\rho_{F/k}: C_k \rightarrow \mathrm{Gal}(F/k)$ be the homomorphism defined by composing $\rho_k$ with the natural projection 
$\mathrm{Gal}(k^{\mathrm{ab}}/k)\to \mathrm{Gal}(F/k)$. Then we have 

$(1)$ $\mathfrak{p}$ is completely decomposed in $F/k$ if and only if $\rho_{F/k}\circ \iota_\mathfrak{p}(k_{\mathfrak{p}}^\times)= \{1\}$, 

$(2)$ $\mathfrak{p}$ is unbranched in $F/k$ if and only if $\rho_{F/k}\circ \iota_\mathfrak{p}(U_{\mathfrak{p}})= \{1\}$.

\end{prop}


\section{Hilbert Theory for 3-manifolds}

\noindent
\indent

In this section, we review a Hilbert Theory for 3-manifolds according to [Mo3, Chap.5]. 
We also show a relation between the linking number and the decomposition law of a knot in a finite Abelian covering, which generalizes a result in [Mo3, Chap.5]. \\

Let $M$ be an integral homology 3-sphere, namely $M$ be a oriented colosed 3-manifold and $H_i(M) \cong H_i(S^3)$ for each $i \in\mathbb{Z}$, 
and let $h:N\rightarrow M$ be a finite Galois covering of connected oriented closed 3-manifolds branched over a link $L \subset M $.
Let $X_L := M\setminus L$, $Y_L := N\setminus h^{-1}(L)$, 
and let $n$ denote the covering degree of $Y_L$ over $X_L$ so that $n =\# \mathrm{Gal}(Y_L/X_L) = \# \mathrm{Gal}(N/M)$. 
Let $K$ be a knot in $M$ which is a component of $L$ or disjoint from $L$, and suppose $h^{-1}(K)= K_1 \cup \cdots \cup K_r$.
Gal$(N/M)$ acts transitively on the set of knots $S_K:=\{\,K_1,\ldots,K_r\,\}$ lying over $K$.
We call the stabilizer $D_{K_i}$ of $K_i$ the $decomposition\ group$ of $K_i$:
$$
D_{K_i}:=\{\, g\in \mathrm{Gal}(N/M) \ |\ g(K_i)=K_i \,\}.
$$
Sinse we obtain the bijection $\mathrm{Gal}(N/M)/D_{K_i} \cong S_K$ for each $i$, 
$\#D_{K_i}=n/r$ is independent of $K_i$.

Since each $g \in \mathrm{Gal}(N/M)$ induces a homeomorphism $g|_{\partial V_{K_i}}:\partial V_{K_i} \rightarrow \partial V_{g(K_{i})}$, 
$g|_{\partial V_{K_i}}$ is a covering transformation of $\partial V_{K_i}$ over $\partial V_K$, 
so we have following isomorphism,
$$
D_{K_i} \cong \mathrm{Gal}(\partial V_{K_i}/\partial V_K).
$$
The Fox completion of the subcovering space of $Y_L$ over $X_L$ corresponding to $D_{K_i}$ is called the $decomposition\ covering\ space$ of $K_i$ and this space is denoted by $Z_{Ki}$.
The map $g \mapsto \bar{g}:=g|_{\partial V_{K_i}}$ induces the homomorphism
$$D_{K_i} \rightarrow \mathrm{Gal}(K_i/K)$$
whose kernel is called the $inertia\ group$ of $K_i$ and is denoted by $I_{K_i}$:
$$I_{K_i}:=\{\, g \in D_{K_i} \ |\ \bar{g}=\text{id}_{K_i} \,\}. $$
If $K_j=g(K_i)$ $(g\in \mathrm{Gal}(N/M))$, we obtain $I_{K_j}=gI_{K_i}g^{-1}$ and hence $\#I_{K_i}$ is independent of $K_i$.
Set $e = e_K := \# I_{K_i}$. 
The Fox completion of the subcovering space of $Y_L$ over $X_L$ corresponding to $I_{K_{i}}$ is called the $inertia\ covering\ space$ of $K_{i}$ and denoted by $T_{K_{i}}$:
\begin{center}
$\xymatrix{N \ar[r] & T_{K_{i}} \ar[r] & Z_{K_{i}} \ar[r] & M}$\\
$\xymatrix{ \{1\} \ar@{-}[r]^{e} & I_{K_{i}} \ar@{-}[r]^{f} & D_{K_{i}} \ar@{-}[r]^(0.4){r} & \mathrm{Gal}(N/M)}$.
\end{center}

Here we have the equalities
$$
\#D_{K_i}=ef, \ \#I_{K_i}=e,\  \#\mathrm{Gal}(K_i/K)=:f.
$$
By comparing the orders, we see that the homomorphism $D_{K_i} \rightarrow \mathrm{Gal}(K_i/K)$ is surjective:
$$1\rightarrow I_{K_i} \rightarrow D_{K_i} \rightarrow \mathrm{Gal}(K_i/K) \rightarrow 1 \ \ \ \ \ \ \ \ \  \text{(exact)}.$$

\par Suppose $h:N \to M$ is an Abelian covering. Then $D_{K_i}$ and $I_{K_i}$, are independent of $K_i$ lying over $K$ and so we denote them by $D_K$ and $I_K$ respectively.

\begin{thm}[Mo3, Chap.5]　Let the notations be as above and suppose $h:N\to M$ is an Abelian covering.
Then we have the exact sequence 
$$1\rightarrow I_{K} \rightarrow D_{K} \rightarrow \mathrm{Gal}(K_i/K) \rightarrow 1 $$
and the equality 
$$ n= efr.$$
\end{thm}

Finally, let us extend the relation between linking number and the decomposition law of a knot in a finite Abelian covering.

\begin{prop}
Let $L:=K_1 \cup \cdots \cup K_r$ be an r-component link in an integral homology $3$-sphere $M$. 
For  given integers $n_i \geq 2$, 
let $\psi: \pi_1(X_L) \rightarrow \mathbb{Z}/n_1\mathbb{Z} \times \cdots \times \mathbb{Z}/n_r\mathbb{Z}$ 
be the homomorphism sending a each meridian of $K_i$ to $(0,\ldots,0,\overset{i}{\check{1}},0,\ldots,0)$. 
Let $Y_L \to X_L$ be the covering corresponding to Ker($\psi$), whose covering degree is $n:=n_1n_2\cdots n_r$, and let $h:N \to M$ be its Fox completion.
Then, for a knot $K$ in $M$ disjoint from $L$, the covering degree of $K$ in $h:N \to M$ coincides with the order of
$(\mathrm{lk}(K,K_1)\ \mathrm{mod}\ n_1,\ldots,\mathrm{lk}(K,K_i)\ \mathrm{mod}\ n_i,\ldots,\mathrm{lk}(K,K_r)\ \mathrm{mod}\ n_r)$ in $\mathbb{Z}/n_1\mathbb{Z} \times \cdots \times \mathbb{Z}/n_r\mathbb{Z}$. 
\end{prop}

\begin{proof}
Let $K'$ be a component of $h^{-1}(K)$.
Since $I_{K'}=I_K=\{1\}$, by Theorem 3.1, the covering degree of $K$ in $h:N\to M$ is the order of a generator $\sigma_{K}$ of $\mathrm{Gal}(K'/K) \cong D_K$ in $\mathrm{Gal}(N/M)$,
where $\sigma_{K}$ corresponds to a loop $K$.
Since $[K]$ is sent to $(\mathrm{lk}(K,K_1)\ \mathrm{mod}\ n_1,\ldots, \mathrm{lk}(K,K_r)\ \mathrm{mod}\ n_r)$ by 
the natural homomorphism $H_1(X_L) \to \mathrm{Gal}(N/M) \cong \mathbb{Z}/n_1\mathbb{Z} \times \cdots \times \mathbb{Z}/n_r\mathbb{Z}$
given by the Hurewicz map and Galois theory, our assertion follows.
\end{proof}

In particular, suppose $K$ is not component of $L$, so that $K$ is unbranched in $N$.
Then the equality $fr=n$ implies that $K$ is decomposed completely in $N$ (i.e. decomposed into an $n$-component link) if and only if for each $i$, $\mathrm{lk}(K_i,K)\equiv 0$ mod $n_i$.

\section{Local class field theory for tori}

\noindent
\indent

In this section, we present a topological analogue of local class field theory for 2-dimensional tori. \\

\par Let $K$ be a fixed knot in an orientable 3-manifold and let $V_K$ be a tubular neighborhood of $K$. 
Let $T_K$ be the boundary of $V_K$, $T_K=\partial V_K$.
Then $T_K$ is a 2-dimensional torus. 
According to (1.3), $T_K$ and $V_K$ are regarded as analogues of a $\mathfrak{p}$-adic local field $k_\mathfrak{p}$ and 
the integer ring $\mathcal{O}_k$.
Let $m$ and $l$ be a meridian and a longitude on $T_K$, respectively.
The inclusion $T_K \hookrightarrow V_K$ induces the homomorphism
$v_K: H_1(T_K) \rightarrow H_1(V_K)=\mathbb{Z}[l]$ whose kernel is $\mathbb{Z}[m]$. 
Thus we have the exact sequence
$$
0 \longrightarrow \mathbb{Z}[m] \longrightarrow H_1(T_K) \longrightarrow \mathbb{Z}[l] \longrightarrow 0 \leqno{(4.1)}
$$
which may be regarded as an analogue of the exact sequence (2.1). \\

\par Let $T_K^{\mathrm{ab}}$ be the maximal Abelian covering of $T_K$ (which is the universal covering).
Since $T_K \cong V_K \setminus K$, (unramified) coverings of $T_K$ correspond to ramified covering of $V_K$ along $K$.
Let $T_K^{\mathrm{ur}}$ be the maximal covering of $T_K$ which comes from the maximal (unramified) covering of $V_K$.
Then we have the following theorem which may be regarded as an analogy of Theorem \ref{local CFT}.

	\begin{thm}[Local class field theory for tori] \label{local CFT for tori}  
		There is a canonical isomorphism $$\rho_{T_K}:H_1(T_K) \rightarrow \mathrm{Gal} (T_K^{\mathrm{ab}}/T_K)$$ which satisfies following properties:
                 
		$(1)$ For any finite Abelian covering $h:  R \rightarrow T_K$, $\rho_{T_K}$ induces the isomorphism
 		$$ 
                	\rho_{R/T_K}: H_1(T_K)/h_*(H_1(R)) \cong \mathrm{Gal}(R/T_K).
                $$

		$(2)$ We have the following commutative diagram 
                 \[\xymatrix{
			0 \ar[r] & \mathbb{Z}[m] \ar[r] \ar[d]^{\rho_{T_K} \vert_{\mathbb{Z}[m]}} & H_1(T_K) \ar[d]^{\rho_{T_K}} \ar[r]^{v_K}  & \mathbb{Z}[l] \ar[d]^{} \ar[r] & 0 \\
			0 \ar[r] & \mathrm{Gal}(T_K^{\mathrm{ab}}/T_K^{\mathrm{ur}}) \ar[r] & \mathrm{Gal}(T_K^{\mathrm{ab}}/T_K) \ar[r] & \mathrm{Gal}(T_K^{\mathrm{ur}}/T_K) \ar[r] & 0 \\
			}\]
                        where horizontal sequences are exact and vertical maps are all isomorphisms.
	\end{thm}
        
        \begin{proof}
Let $\eta_{T_K}:H_1(T_K) \to \pi_1(T_K)/[\pi_1(T_K),\pi_1(T_K)]$ be the Hurewicz isomorphism, where $[\pi_1(T_K),\pi_1(T_K)]$ is commutator subgroup of $\pi_1(T_K)$.
We define $\rho_{T_K}:H_1(T_K) \rightarrow \mathrm{Gal} (T_K^{\mathrm{ab}}/T_K)$ by the composite of $\eta_{T_K}$ with the isomorphism 
$\pi_1(T_K)/[\pi_1(T_K),\pi_1(T_K)] \cong \mathrm{Gal}(T_K^{\mathrm{ab}}/T_K)$ coming from Galois theory of covering spaces.


(1) Since $h:R\rightarrow T_K$ is the finite Abelian covering, $R$ is torus. 
By Galois theory we  have $\pi_1(T_K)/h_*(\pi_1(R)) \cong \mathrm{Gal}(R/T_K)$. 
Since $\pi_1(T_K)$ and $\pi_1(R)$ are Abelian groups, we have $\pi_1(T_K)=H_1(T_K)$ and $\pi_1(R)=H_1(R)$. 
Hence $\rho_{T_K}$ induces the isomorphism $H_1(T_K)/h_*(H_1(R)) \cong \mathrm{Gal}(R/T_K)$.

(2) The upper horizontal exact sequence is nothing but (4.1). The lower horizontal sequence is coming from Galois theory.
First, we obtain the following isomorphism:
	\begin{eqnarray*}
\mathrm{Gal}(T_K^{\mathrm{ur}}/T_K) & \cong & H_1(T_K)/h_*(H_1(T_K^{\mathrm{ur}})) \\
					     & \cong & (\mathbb{Z}[m] \times \mathbb{Z}[l])/(\mathbb{Z}[m] \times 0) \\
  				      	     & \cong & \mathbb{Z}[l].
	\end{eqnarray*}
Then we consider following isomorphism, $ \mathrm{Gal}(T_K^{\mathrm{ab}}/T_K^{\mathrm{ur}}) \cong \pi_1(T_K^{\mathrm{ur}})/h_*(\pi_1(T_K^{\mathrm{ab}}))$.
Since $\pi_1(T_K)$ is Abelian group, $T_K^{\mathrm{ab}}$ is the universal covering. 
Therefore we have $\mathrm{Gal}(T_K^{\mathrm{ab}}/T_K^{\mathrm{ur}}) \cong \pi_1(T_K^{\mathrm{ur}}) = H_1(T_K^{\mathrm{ur}})$.
By the construction of $T_K^{\mathrm{ur}}$, $T_K^{\mathrm{ur}}$ is homeomorphic to $S^1\times \mathbb{R}$, whose $S^1$ corresponds to a meridian on $T_K$.
Hence $H_1(T_K^{\mathrm{ur}}) \cong \mathbb{Z}[m]$.
        \end{proof}
        
 \begin{cor}\label{cor of local CFT for tori} 
 There is the one to one correspondence between the set of finite unbranched coverings of $V_K$ and 
 the set of finite index subgroups of $H_1(T_K)$ containing $\mathbb{Z}[m]$.
 \end{cor}
 
 \begin{proof}
 The commutative diagram of (2) implies the corollary. 
 \end{proof}

\begin{defn} Since we have $H_1(T_K)=\mathbb{Z}[m]\times \mathbb{Z}[l]$, 
we can write an element $a\in H_1(T_K)$ as $a=(q,p) \in \mathbb{Z}^2$ if $a=q[m]+p[l]$. 
We call $q$ the {\it meridian component} of $a$ and $p$ the {\it longitude component} of $a$. 
The longitude component represents the value of $v_K$.
\end{defn}


\section{Global class field theory for 3-manifolds}
\noindent
\indent

In this section, let $M$ be a closed orientable 3-manifold, and for a certain set $\mathcal{K}$ of knots in a 3-manifold $M$, 
we introduce the id\`ele group and id\`ele class group, and the global reciprocity homomorphism,
by getting the local theory in the section 4 together over all knots in $\mathcal{K}$.
We then establish an analogue of the isomorphism theorem in global class field theory.
Now, we define a set of knots $\mathcal{K}$.

	\begin{defn} 
		We call a set $\mathcal{K}$ of knots in $M$ {\it admissible}, if $\mathcal{K}$ satisfies the following conditions:

		(1) For each $K_i \in \mathcal{K}$, there exists a tubular neighborhood $V_{K_i}$ such that  $V_{K_i}\cap V_{K_j}=\emptyset $ if $K_{i}$, $K_{j} \in \mathcal{K}$ and $K_i \neq K_j$.
                
                (2) $\mathcal{K}$ contains generators of $H_1(M)$.
                
                (3) $\#\mathcal{K} = \# \mathbb{N}$.
	\end{defn}
        
In this paper, we fix such an admissible set $\mathcal{K}$ of knots in $M$ and consider a pair $(M;\mathcal{K})$. 
And we identify with $\pi_1^{\mathrm{ab}}(X) \cong \mathrm{Gal}(X^{\mathrm{ab}}/X) \cong H_1(X)$ for any topological space $X$. 

	\begin{defn}[id\`ele group] We define the {\it id\`ele group} of $(M;\mathcal{K})$ by　
		$$I_{(M;\mathcal{K})}:=\left\{\, (a_K)_{K} \in \displaystyle\prod_{K \in \mathcal{K}}H_1(\partial V_K) \,|\, v_{K}(a_K)=0 \text{ for almost all } K \in \mathcal{K} \,\right \}.$$
	\end{defn}　\\


Let $\mathcal{L}$ be a set of finite subsets of $\mathcal{K}$.
We define the order for $\mathcal{L}$ in the following way, if $L_\alpha , L_\beta \in \mathcal{L}$, $L_\alpha \leq L_\beta \xLeftrightarrow{def} L_\alpha \subset L_\beta$.
In this paper, for a finite subset $L \subset \mathcal{K}$, we denote by $X_L$ the exterior space $M \setminus L$.
Then for each $L_\alpha \leq L_\beta$, we consider the homomorphism
$$\varphi_{\alpha\beta}:\mathrm{Gal}(X_{L_\beta}^{\mathrm{ab}}/X_{L_\beta}) \rightarrow \mathrm{Gal}(X_{L_\alpha}^{\mathrm{ab}}/X_{L_\alpha})$$
which is induced by natural inclusion $\iota_{\alpha\beta}:X_{L_{\beta}} \hookrightarrow X_{L_{\alpha}}$.

When $L_\alpha \leq L_\beta \leq L_\gamma$, it satisfies $\varphi_{\alpha\gamma} = \varphi_{\alpha\beta}\circ \varphi_{\beta\gamma}$ from $\iota_{\alpha\gamma} = \iota_{\alpha\beta}\circ \iota_{\beta\gamma}$.
Then we define $\mathrm{Gal}(M;\mathcal{K})^{\mathrm{ab}}$ by the inverse limit of the Galois groups $\mathrm{Gal}(X_{L_\alpha}^{\mathrm{ab}}/X_{L_\alpha})$ with respect to $L_\alpha$:
$$
\mathrm{Gal}(M;\mathcal{K})^{\mathrm{ab}}
:=\varprojlim_{\alpha}\mathrm{Gal}(X_{L_\alpha}^{\mathrm{ab}}/X_{L_\alpha})
:=\left\{\,(a_\alpha)_\alpha \in \displaystyle\prod_{L_{\alpha} \in \mathcal{L}}\mathrm{Gal}(X_{L_\alpha}^{\mathrm{ab}}/X_{L_\alpha}) \,|\, \varphi_{\alpha\beta}(a_\beta)=a_\alpha \,\right\}.
$$
This group may be regarded as an analogue of the maximal Abelian Galois group Gal$(k^{\mathrm{ab}}/k)$ of a number field $k$. 

\par Now, we are going to define an analogue of the global reciprocity homomorphism
$$
\rho_M:I_{(M;\mathcal{K})} \rightarrow \mathrm{Gal}(M;\mathcal{K})^{\mathrm{ab}}
\leqno{(5.1)}
$$
as follows.
Firstly, for each $\alpha$ and $K \in \mathcal{K}$, the natural inclusion
$\iota_{K}^{\alpha}:\partial V_K \to X_{L_\alpha}$ induces the homomorphisms 
$\iota_{K*}^{\alpha}:H_1(\partial V_K) \to H_1(X_{L_\alpha})$ and 
$g_K^{\alpha}:\mathrm{Gal}(\partial V_K^{\mathrm{ab}}/\partial V_K) \to \mathrm{Gal}(X_{L_\alpha}^{\mathrm{ab}}/X_{L_\alpha})$ which fit in the commutative diagram

\begin{equation}
\vcenter{
                \xymatrix{
			H_1(\partial V_K) \ar[d]_{\iota_{K*}^\alpha} \ar[r]^(0.4){\rho_{\partial V_K}} \ar@{}[dr]|\circlearrowleft & \mathrm{Gal}(\partial V_K^{\mathrm{ab}}/\partial V_K) \ar[d]^{g_K^\alpha} \\
			H_1(X_{L_\alpha}) \ar[r]_(0.4){\eta_{X_{L_\alpha}}} 							& \mathrm{Gal}(X_{L_\alpha}^{\mathrm{ab}}/X_{L_\alpha}) \\
		}} \tag{5.2}
\end{equation}

when $\eta_{X_{L_\alpha}}$ is the isomorphism by the Hurewicz map. 
We let
$$
\lambda_K^\alpha:H_1(\partial V_K) \to \mathrm{Gal}(X_{L_\alpha}^{\mathrm{ab}}/X_{L_\alpha})
$$
be the composite $g_K^\alpha \circ \rho_{\partial V_K} = \eta_{X_{L_\alpha}} \circ \iota_{K*}^\alpha$ and define
$$
\psi_\alpha:I_{(M;\mathcal{K})} \to \mathrm{Gal}(X_{L_\alpha}^{\mathrm{ab}}/X_{L_\alpha})
$$
by
$$
\psi_\alpha((a_K)_K):=\displaystyle \sum_{K\in\mathcal{K}}\lambda_K^\alpha(a_K).
$$

Here the summation over $\mathcal{K}$ is finite, because 
longitude component of $a_K$ is 0 for almost all $K\in\mathcal{K}$ 
and the meridian component of $a_K$ is 0 in $H_1(X_{L_\alpha})$ for $K \notin  L_\alpha$.

Finally, we define the global reciprocity homomorphism $\rho_M:I_{(M;\mathcal{K})} \to \mathrm{Gal}(M;\mathcal{K})^{\mathrm{ab}}$ by
$$
\rho_M((a_K)_K):=(\psi((a_K)_K))_\alpha,
$$
noticing $\varphi_\alpha = \varphi_{\alpha\beta}\circ \psi_{\beta}$,

	\begin{eqnarray*}
\varphi_{\alpha\beta}\circ\psi_{\beta}((a_K)_K) & = & \varphi_{\alpha\beta}(\sum_{K\in\mathcal{K}}\iota_{K_*}^\beta (a_K)) \\
					     & = & \sum_{K\in\mathcal{K}}\iota_{\alpha\beta_{*}} \circ \iota_{K_*}^\beta (a_K)) \\
  				      	     & = & \sum_{K\in\mathcal{K}}\iota_{K_*}^\alpha (a_K) \\
                                             & = & \psi_\alpha((a_K)_K)
	\end{eqnarray*}
where each $\iota$ is induced by following diagram:

\[\xymatrix{
\partial V_K \ar[d]_{\iota_{K}^{\beta}} \ar[dr]^{\iota_{K}^{\alpha}} & \\
X_{L_\beta} \ar[r]^{\iota_{\alpha\beta}} & X_{L_\alpha}. \\
}\]

We define a principal id\`ele group, and id\`ele class group of manifolds as follows.

\begin{defn}[Principal id\`ele group, id\`ele class group] 
We define the {\it principal id\`ele group} of $(M;\mathcal{K})$ by
$$P_{(M;\mathcal{K})}:=\left\{\,(a_K)_K \in I_{(M;\mathcal{K})} \,|\,
 \displaystyle \sum_{K\in\mathcal{K}}\iota_{K_*}^\alpha (a_K) = 0 \in H_1(X_{L_\alpha}) \text{ for all finite subset } L_\alpha \subset \mathcal{K} \,\right\},$$ 
and define the {\it id\`ele class group} of $(M;\mathcal{K})$ by
$C_{(M;\mathcal{K})}:= I_{(M;\mathcal{K})} / P_{(M;\mathcal{K})}$.
\end{defn}

We note that our definitions of the id\`ele group, principal id\`ele group and 
id\`ele class group depend on a choice of an admissible set $\mathcal{K}$ of knots in $M$.

\begin{lem}
There exists an admissible set $\mathcal{K}$ of knots in $M$ which satisfies the following conditions:
(A) For any finite subset $L$ of $\mathcal{K}$ and any finite cover $h:N \rightarrow M$ branched over $L$, 
$h^{-1}(\mathcal{K})$ is an admissible set of knots in $N$.
\end{lem}

\begin{proof}
We note firstly that there are only countably many isotopy classes of links in $M$ 
and that there are only countably many closed 3-manifolds.
Let $\mathcal{L}$ be the set of finite subsets of knots in $M$.
For each $L \in \mathcal{L}$, there are only countably many covers of $M$ branched over $L$,
say $\mathcal{C}_L=\{\, h_L:N_L \to M \,\}$. 
Choose a finite set of knots in $N_L$, $\{\, K_{N_L}^\mu \  |\  1\leq \mu \leq m_{N_L} \,\}$, 
such that their homology classes $[K_{N_L}^\mu]$ generate $H_1(N_L)$.
We then define $\mathcal{K}$ to be the union of $h(K_{N_L}^\mu)$ over all $N_L \in \mathcal{C}_L$ and $L\in \mathcal{L}$:
$$
\mathcal{K}:=\{\, h(K_{N_L}^\mu) \,|\, [K_{N_L}^\mu]\text{'s generate } H_1(N_L),\ N_L\in\mathcal{C}_L,\  L\in \mathcal{L} \,\}.
$$

Here, by moving knots $K_{N_L}^\mu$ a little bit if necessary, we may assume that if $K_{N_L}^\mu$ and $K_{N'_{L'}}^\nu$ are distinct,
$K_{N_L}^\mu \cap K_{N'_{L'}}^\nu$ and $h(K_{N_L}^\mu) \cap h(K_{N'_{L'}}^\nu)$ are empty and 
that any $K_{N_L}^\mu$ does not intersect a branch set of $N_L$.
Then the conditon (A) is satisfied by our construction.
\end{proof}

\begin{defn}[very admissible set] 
We call an admissible set $\mathcal{K}$ of knots in $M$ {\it very admissible} if $\mathcal{K}$ satisfies the condition in Lemma 5.4.
\end{defn}

Hereafter we fix such a good admissible set $\mathcal{K}$ of knots once and for all. 
So we assume that $h^{-1}(\mathcal{K})$ is admissible set of Abelian covering space $N$, which is branched over $L\subset \mathcal{K}$.
We set $\mathcal{K}_N :=h^{-1}(\mathcal{K})$. 
And its order is induced by $\mathcal{K}$, namely if $L_\alpha \leq L_\beta$ in $\mathcal{K}$, then $h^{-1}(L_\alpha) \leq h^{-1}(L_\beta)$ in $\mathcal{K}_N$. \\

Next, we define the norm map $h_{N/M}:I_{(N;\mathcal{K}_N)} \rightarrow I_{(M;\mathcal{K})}$.
Let $h^{-1}(K)=K_1 \cup K_2 \cup \cdots \cup K_r$ for each $K \in \mathcal{K}$.
For a tubular neighborhood $V_K$ of $K$, let $V_{K_i}$ be a connected component of $h^{-1}(V_K)$ containing $K_i$.
Let $h_{i_*}:H_1(\partial V_{K_i}) \rightarrow H_1(\partial V_K)$ be the homomorphism which is induced by $h_i:=h|_{\partial V_{K_i}}:\partial V_{K_i} \rightarrow \partial V_K$.
We then define $h_K:\displaystyle \bigoplus_{i=1}^{r} H_1(\partial V_{K_i}) \rightarrow H_1(\partial V_K)$ by 
$h_K((a_i)_{i=1}^{r}):=\sum_{i=1}^r h_{i_*}(a_{i})$ and $h_{N/M}:I_{(N;\mathcal{K}_N)} \rightarrow I_{(M;\mathcal{K})}$ is defined by $h_{N/M}:=\sum_{K\in \mathcal{K}}h_K$. 
The norm map $h_{N/M}$ induces the homomorphism $\bar{h}_{N/M}:C_{(N;\mathcal{K}_N)} \rightarrow C_{(M;\mathcal{K})}$. 
We will write $\bar{h}_{N/M}$ simply $h_{N/M}$ when no confusion can arise.

\begin{rem}
We note that $h_{N/M}:C_{(N;\mathcal{K}_N)} \rightarrow C_{(M;\mathcal{K})}$ is well defined, because following diagram is commutative:
\[\xymatrix{
I_{(N;\mathcal{K}_N)} \ar[d] \ar[r]^(0.4){\rho_N} \ar@{}[dr]|\circlearrowleft & \mathrm{Gal}(N;\mathcal{K}_N)^{\mathrm{ab}} \ar[d] \\
I_{(M;\mathcal{K})} \ar[r]_(0.4){\rho_M} & \mathrm{Gal}(M;\mathcal{K})^{\mathrm{ab}}. \\
}\]　 
\end{rem}

Next, we show an analogue of the relation between id\`ele class group and ideal class group in our 3-manifold context.
We define the following homomorphism,
$$
\xi:  I_{(M;\mathcal{K})} \to \displaystyle \bigoplus _{K \in \mathcal{K}}\mathbb{Z}K 
$$
by
$$
   \xi((a_K)_K) := (v_K(a_K))_K.
$$
Then, we denote Ker$(\xi)$ by $U$.

\begin{prop}
Assume that a very admissible set $\mathcal{K}$ satisfies the following condition: 
For any finite subset $\{K_j\}_{j\in J}$ of $\mathcal{K}$, 
if one has $\displaystyle \sum_{j\in J}c_j[K_j]=0$ in $H_1(M)$ with $c_j \in \mathbb{Z}\setminus \{0\}$, then $[K_j]=0$. 
Then we have 
$$H_1(M) \cong I_{(M;\mathcal{K})}/ (U+P_{(M;\mathcal{K})}).$$
\end{prop}

\begin{proof}
For $K\in \mathcal{K}$, let $\iota_K^M:\partial V_K \to M$ be the inclusion map and $\iota_{K*}^M:H_1(\partial V_K) \to H_1(M)$ the induced homomorphism
$$
\varphi: I_{(M;\mathcal{K})} \rightarrow H_1(M)
$$
by
$$
\varphi((a_K)_K) :=\displaystyle \sum_{K \in \mathcal{K}}\iota_{K*}^M(a_K).
$$
Here the summation over $K \in\mathcal{K}$ is actually a finite sum, 
because $v_K(a_K)=0$ for almost all $K \in\mathcal{K}$ and a meridian of any $K$ is null-homologous in $M$.
It is easy to see that $\varphi$ is surjective, using to the condition (2) of Definition 5.1 of $\mathcal{K}$.
Therefore it suffices to show Ker$(\varphi)=U+P_{(M;\mathcal{K})}$.

\par Suppose $(a_K)_K \in U+P_{(M;\mathcal{K})}$. 
Then we can write $(a_K)_K=(b_K)_K+(c_K)_K$ with $(b_K)_K \in P_{(M;\mathcal{K})}$ and $(c_K)_K \in U$.
By Definition 5.3 of $P_{(M;\mathcal{K})}$, it is easy to see $\varphi((b_{K})_K)=0$, hence $(b_K)_K \in \mathrm{Ker}(\varphi)$.
As for $(c_K)_K$, we also have $\varphi((c_K)_K)=0$, because $v_K(c_K)=0$ for all $K\in \mathcal{K}$ and
a meridian of any $K \in \mathcal{K}$ is null homologous in $M$. 
Therefore we have $(a_K)_K \in \mathrm{Ker}(\varphi)$.

\par Suppose $(a_K)_K \in \mathrm{Ker}(\varphi)$. As in Definition 4.3, we decompose $a_K$ to the meridian and longitude components:
$$
a_K=(q_K,p_K)=q_K[m_K] + p_K[l_K],
$$
where $m_K$ is the meridian of $K$, $l_K$ is a longitude of $K$, and $p_K, q_K \in \mathbb{Z}$.
If $p_K=0$ for all $K\in \mathcal{K}$, then $(a_K)_K\in U$ and we are done.
So we may assume there are some $K$, say $K_1,\ldots,K_n$, such that $p_{K_1},\ldots,p_{K_n}\neq 0$. We write

$$\left\{
\begin{array}{l}
 (a_K)_K=(b_K)_K+(c_K)_K  \\
 (b_K)_K=(\ldots,\bm{0},(0,p_{K_1}),\bm{0},\ldots,\bm{0},(0,p_{K_n}),\bm{0},\ldots) \\
 (c_K)_K\in U,
\end{array}
\right.
$$
where $\bm{0}=(0,0)$. Then it suffices to show $(b_K)_K \in P_{(M;\mathcal{K})}+U$.
Since $\varphi((a_K)_K)=\varphi((c_K)_K)=0$, we have 
$$
0=\varphi((b_K)_K)=\displaystyle \sum_{i=1}^{n}p_{K_i}[K_i].
$$
By our assumption, we have $[K_i]=0$ in $H_1(M)$ for $i=1,\ldots,n.$
Therefore there exists a surface $S$ such that $\partial S = K_1\cup\cdots\cup K_n$.
If there is no $K\in\mathcal{K}\setminus \{ K_1,\ldots,K_n \}$ which intersects with $S$,
then $(b_K)_K \in P_{(M;\mathcal{K})}$.
Suppose there is $\{ K_\mu \} \subset \mathcal{K}$ such that $K_\mu$ intersects with $S$. 
Then let $n_{i\mu}:=\mathrm{lk}(K_i,K_\mu)$ be the linking number of $K_i$ and $K_\mu$ which can be defined by our assumption, and we have
$$
(\ldots,\bm{0},(0,p_{K_i}),\bm{0},\ldots) = (\ldots,\bm{0},(0,p_{K_i}),\bm{0},\ldots,\bm{0},(-p_{K_i}n_{i\mu},0),\bm{0},\ldots) + (\ldots,\bm{0},(p_{K_i}n_{i\mu},0),\bm{0},\ldots).
$$
Since the first term of the right hand side is in $P_{(M;\mathcal{K})}$ and the second one is in $U$, we are done.


\end{proof}

\begin{cor}
Let $M$ be an integral homology sphere.
Then $I_{(M;\mathcal{K})} = U+P_{(M;\mathcal{K})}$.
\end{cor}

\begin{proof}
Since $H_1(M) = 0$, a very admissible set $\mathcal{K}$ of $M$ satisfies the assumption of Proposition 5.7.
Then $I_{(M;\mathcal{K})}/ (U+P_{(M;\mathcal{K})})=0$ by Proposition 5.7.
\end{proof}


Finally, we present our main result on global class field theory for integral homology sphere.

	\begin{thm}[Global class field theory over an integral homology sphere]\label{global CFT for mfd}$\\$ 
        Let M be an integral homology sphere and let $\mathcal{K}$ be a very admissible set of knots in $M$.
        Then, there exists homomorphism, $$\rho_{M}:C_{(M;\mathcal{K})} \rightarrow \mathrm{Gal}(M;\mathcal{K})^{\mathrm{ab}}$$ which has the following properties:

		$(1)$ For any finite Abelian covering  $h:N \rightarrow M$ branched over $L$, $\rho_M$ induces the isomorphism,
		$$
               		C_{(M;\mathcal{K})}/h_{N/M}(C_{(N;\mathcal{K}_N)})\cong \mathrm{Gal}(N/M)
		$$
                where $h_{N/M}$ denotes the norm map on the id\`ele class group.

		$(2)$
                For a knot $K \in \mathcal{K}$, we have the following commutative diagram:
                \[\xymatrix{
			H_1(\partial V_K) \ar[d]_{\iota_K} \ar[r]^(0.4){\rho_K} \ar@{}[dr]|\circlearrowleft & \mathrm{Gal}(\partial V_K^{\mathrm{ab}}/\partial V_K) \ar[d]^{g_K} \\
			C_{(M;\mathcal{K})} \ar[r]_(0.4){\rho_M} 							& \displaystyle \mathrm{Gal}(M;\mathcal{K})^{\mathrm{ab}}\\
		}\]
                where $\iota_K$ is the homomorphism induced by the natural inclusion $H_1(\partial V_K) \rightarrow I_{(M;\mathcal{K})}$. Namely, $\iota_K(a_K)=[(\ldots,\bm{0},a_K,\bm{0},\ldots)]$, 
                and $g_K$ is the homomorphism induced by $g_K^\alpha$ in (5.2).
	\end{thm}

	\begin{proof}
We note that a very admissible set $\mathcal{K}$ exists by lemma 5.4.
Let $M\setminus L =X_L$, $N\setminus h^{-1}(L)=Y_L$, $L=K_1 \cup \cdots \cup K_r$. 
Since $h:N \rightarrow M$ is Abelian covering branched over $L$, $\mathrm{Gal}(N/M) \cong H_1(X_L)/h_*(Y_L)$.

We define $\rho_M$ as the canonical induced homomorphism by (5.1). 
We note that this is well-defined by definition of $P_{(M;\mathcal{K})}$.

(1) We have 
\begin{eqnarray*}C_{(M;\mathcal{K})}/ h_{N/M}(C_{(N;\mathcal{K}_N)}) &=& (I_{(M;\mathcal{K})}/P_{(M;\mathcal{K})}) / h_{N/M}(I_{(N;\mathcal{K}_N)}/P_{(N;\mathcal{K}_N)}) \\
 &=& I_{(M;\mathcal{K})} / (P_{(M;\mathcal{K})} + h_{N/M}(I_{(N;\mathcal{K}_N)})) \\
 &=& (P_{(M;\mathcal{K})} + U) / (P_{(M;\mathcal{K})} + h_{N/M}(I_{(N;\mathcal{K}_N)})) \ \ \ \ \ \ \ \ \ \ \text{by Corollary 5.8}\\
 &=& U/ U \cap (P_{(M;\mathcal{K})} + h_{N/M}(I_{(N;\mathcal{K}_N)})). \end{eqnarray*}
 
Therefore, we need to prove $U/ U \cap (P_{(M;\mathcal{K})} + h_{N/M}(I_{(N;\mathcal{K}_N)})) \cong \mathrm{Gal}(N/M) \cong H_1(X_L)/h_*(Y_L)$.

Let us define the homomorphism $\varphi: U \to H_1(X_L)/h_*(Y_L)$ by $\varphi((a_K)_K):= \displaystyle \pi \circ \sum_{K\in \mathcal{K}} \iota_K^L(a_K)$, 
where the homomorphism $\iota_K^L:H_1(T_K) \to H_1(X_L)$ is induced by the natural inclusion $T_K \to X_L$, and $\pi:H_1(X_L) \to H_1(X_L)/h_*(Y_L)$ is the natural projection.

Since $M$ is an integral homology sphere, $H_1(X_L)$ is generated by the meridian classes of $K_i$'s. 
Hence $\varphi$ is surjective by the definition. 

Therefore, it suffices to prove $\mathrm{Ker}(\varphi) = U \cap (P_{(M;\mathcal{K})}+ h_{N/M}(I_{(N;\mathcal{K}_N)})).$

Let $(a_K)_K$ be an element of Ker$(\varphi)$, namely $\sum_{K\in \mathcal{K}} \iota_K^L(a_K) \in h_*(Y_L)$. 
We note that the longitude component of $a_K$ is 0 for each $K \in \mathcal{K}$. 
For each component $K \in \mathcal{K}$, we consider the homomorphism $h_K:\displaystyle \bigoplus_{\tilde{K}\in \pi_0(h^{-1}(K))} H_1(T_{\tilde{K}}) \to H_1(T_K)$.
If $K$ is unbranched component, the definition of unbranch covering implies Im$(h_K) \supset \mathbb{Z}[m] \times 0$.
Therefore it suffices to check that Im$(h_{K_i})\supset \mathbb{Z}[m_i] \times 0$.
Since $H_1(X_L)$ is generated by the meridian classes of $K_i$'s, and the longitude component of $a_{K_i}$ is 0, 
$\sum_{K\in \mathcal{K}} \iota_K^L(a_K) \in h_*(Y_L)$ implies $\iota_{K_i}^L(a_{K_i}) \in h_*(Y_L)$.
It implies $a_{K_i}$ is an element of Im$(h_{K_i})$. 
Therefore $(a_K)_K$ is an element of $U \cap h_{N/M}(I_{(N;\mathcal{K}_N)})$.

Let $(a_K)_K + (b_K)_K$ be an element of $U \cap (P_{(M;\mathcal{K})} + h_{N/M}(I_{(N;\mathcal{K}_N)}))$, 
where $(a_K)_K$ is an element of $P_{(M;\mathcal{K})}$, and $(b_K)_K$ is an element of $h_{N/M}(I_{(N;\mathcal{K}_N)}))$. 
By the definition of $P_{(M;\mathcal{K})}$, $\sum_{K\in \mathcal{K}} \iota_K^L(a_K) =0$. 
Then $\varphi((b_K)_K)=0$ by the definition of norm map $h_{N/M}$. 
Hence $(a_K)_K + (b_K)_K \in \text{Ker}(\varphi)$. 

Thus we obtain $\mathrm{Ker}(\varphi) = U \cap (P_{(M;\mathcal{K})}+ h_{N/M}(I_{(N;\mathcal{K}_N)}))$.

(2) Our assertion follows from 

\begin{eqnarray*}
 g \circ \rho_K(a,b)	&=& (\iota_{K_*}^{\alpha}(a,b))_{L_{\alpha}}
\end{eqnarray*}
and
\begin{eqnarray*}
\rho_M \circ \iota_K (a,b) &=&\rho_M[(0, \ldots , 0 , (a,b) , 0 , \ldots )] \\
&=&(\sum_{K\in\mathcal{K}} \iota_{K_*}^\alpha ((a_K)_K))_{L_\alpha} \\
&=&(\iota_{K_*}^\alpha (a,b))_{L_\alpha}.
\end{eqnarray*}
	\end{proof}

\begin{eg}
Let $M=S^3$, we choose a link $L=K_1 \cup \cdots \cup K_r \in \mathcal{L}$. Let $X_L:= M\setminus L$, and let $m_i$ be a meridian class of $K_i$. 
The map sending each meridian class $m_i$ to 1 defines a surjective homomorphism $\psi: \pi_1(X_L) \to \mathbb{Z}/n_1\mathbb{Z} \times \cdots \times \mathbb{Z}/n_r\mathbb{Z} \ ;\  m_i \mapsto 1 \in \mathbb{Z}/n_i\mathbb{Z}$. 
For an Abelian covering of $X_L$ corresponding to Ker$(\psi)$, we have the Fox completion $N$, which is an Abelian covering of $S^3$ branched over $L$.
Then, the Galois group Gal$(N/S^3) \cong \mathbb{Z}/n_1\mathbb{Z} \times \cdots \times \mathbb{Z}/n_r\mathbb{Z}$.
We are going to show that $C_{(S^3;\mathcal{K})}/h_{N/S^3}(C_{(N;\mathcal{K}_{N})})\cong \mathbb{Z}/n_1\mathbb{Z} \times \cdots \times \mathbb{Z}/n_r\mathbb{Z}$. 
By the above proof, it is sufficiant to show that $U/ U \cap (P_{(S^3;\mathcal{K})} + h_{N/S^3}(I_{(N;\mathcal{K}_{N})}) \cong \mathbb{Z}/n_1\mathbb{Z} \times \cdots \times \mathbb{Z}/n_r\mathbb{Z}$. 


We define the homomorphism $\varphi: U \rightarrow \mathbb{Z}/n_1\mathbb{Z} \times \cdots \times \mathbb{Z}/n_r\mathbb{Z}$ by
$\varphi((a_K)_K) := (m_{K_1}, \cdots ,m_{K_r})$, where $m_{K_i} \in \mathbb{Z}/n_i\mathbb{Z}$ be the meridian component of $H_1(T_{K_i})$.

$\varphi$ is surjective by the definition, so we prove $\mathrm{Ker}(\varphi) = U \cap (P_{(S^3;\mathcal{K})}+ h(I_{(N;\mathcal{K}_N)})).$

Let $(a_K)_K$ be an element of Ker$(\varphi)$. 
Assume $K$ is an unbranched component of $\mathcal{K}$, 
the meridian component of $a_K \in H_1(T_K)$ is coming from a meridian component of some element of $H_1(T_{h^{-1}(K)})$. 
Therefore we move on the branched component $K_i$.
We denote by $K_i'$ one of the connected component of $h^{-1}(K_i)$. 
The meridian component of $m_{K_i}$ is multiples of $n_i$, 
so $a_{K_i} \in H_1(T_{K_i})$ is coming from $H_1(T_{K_i'})$ by the fact that $K_i$ has the branched index $n_i$.
Therefore $(a_K)_K \in U \cap (0 + h(I_{(N;\mathcal{K}_N)}))$.

Let $(a_K)_K$ be an element of $U \cap (P_{(S^3;\mathcal{K})} + h(I_{(N;\mathcal{K}_N)}))$, 
$(a_K)_K$ be the following form 
$$
(a_K)_K = (\ldots,(m_K,l_K),\ldots) + (\ldots,(m_K',-l_K),\ldots).
$$
The first term element is in $P_{(S^3;\mathcal{K})}$, and the second term element is in $h(I_{(N;\mathcal{K}_N)})$. 
Then we set first term by $(b_K)_K$, second term by $(c_K)_K.$
$m_{K_i}'$ is coming from the norm map, so $m_{K_i}$ is multiples of $n_i$.
Therefore we consider $m_{K_i}$.
We define the subset $I := \{\, i \in \{\, 1,2,\ldots,n \,\} \,|\, m_{K_i}\neq 0 \,\}$.

For each $i \in I$, there exist a knot $K_{i\mu}\in \mathcal{K}$, 
such that $m_{K_i}=n_{i\mu}x_{i\mu}$ where $n_{i\mu}:=\mathrm{lk}(K_i,K_{i\mu})$.

i.e. $$(b_K)_K = (\ldots,(x_{i\mu}n_{i\mu},l_{K_i}),\ldots,(-l_{K_i}n_{i\mu},-x_{i\mu}),\ldots,)$$

where $x_{i\mu} \in \mathbb{Z}$.
The form of $(b_K)_K$ and the fact that $(a_K)_K=(b_K)_K+(c_K)_K \in U$ imply the form of $(c_K)_K$, namely
$$
(c_K)_K = (\ldots,( * ,-l_{K_i}),\ldots,(* ,x_{i\mu}),\ldots).
$$

Next, we consider $x_{i\mu}$.
We denote by $K_{i\mu}'$ a connected component of $h^{-1}(K_{i\mu})$. 
Then $f_{i\mu}$ is the covering degree of $K_{i\mu}'$ over $K_{i\mu}$. 
By the Proposition 3.2, $f_{i\mu}$ is the order of the element 
$(\mathrm{lk}(K_{i\mu},K_1),\ldots, \mathrm{lk}(K_{i\mu},K_r)) \in \mathbb{Z}/n_1\mathbb{Z} \times \cdots \times \mathbb{Z}/n_r\mathbb{Z}$.

Hence $f_{i\mu}$ is the multiples of $n_i/d_i$, where $d_i$ is the greatest common divisor of $n_{i\mu}$ and $n_i$.
Therefore $x_{i\mu}$ is the multiple of $n_i/d_i$.
It implies $m_{K_i}= n_{i\mu}x_{i\mu} = n_{i\mu}\frac{n_i}{d_i}y = \frac{n_{i\mu}}{d_i}n_iy \in n_i\mathbb{Z}$.

Since $m_{K_i}$ is the multiple of $n_i$ for each $K_i$, $(a_K)_K$ is an element of kernel. 
Thus we obtain the isomorphism $C_{(S^3;\mathcal{K})}/h_{N/S^3}(C_{(N;\mathcal{K}_{N})})\cong \mathbb{Z}/n_1\mathbb{Z} \times \cdots \times \mathbb{Z}/n_r\mathbb{Z}$.

\end{eg}

At last, we are going to introduce an analogu of Proposition \ref{cor of global CFT}.

\begin{prop}
For a finite Abelian covering $h:N \to M $ branched over $L \in \mathcal{L}$,\\
let $\rho_{N/M}: C_M \rightarrow \mathrm{Gal}(N/M)$ be the homomorphism defined by composing $\rho_M$ with the natural projection 
$\mathrm{Gal}(M;\mathcal{K})^{\mathrm{ab}} \to \mathrm{Gal}(N/M)$. 
Then we have 

$(1)$ $K \in \mathcal{K}$ is completely decomposed in $N$ if and only if $\rho_{N/M} \circ \iota_K(H_1(\partial V_K))= \{1\}$, 

$(2)$ $K \in \mathcal{K}$ is unbranched in $N$ if and only if $\rho_{N/M} \circ \iota_K(\mathbb{Z}[m])= \{1\}$.

\end{prop}

\begin{proof}
(1) Let $K'$ be a connected component of $h^{-1}(K)$. 
By Section 3, we identify with Gal$(\partial V_{K'}/\partial V_K)$ and decomposition group of $K'$. 
By Theorem \ref{global CFT for mfd}, the following diagram is commutative:

\[\xymatrix{
H_1(\partial V_K) \ar[d]_{\rho_{N/M} \circ \iota_K} \ar[r] \ar@{}[dr]|\circlearrowleft & \mathrm{Gal}(\partial V_{K'}/\partial V_K) \ar[d]^{\cap} \\
\mathrm{Gal}(N/M) \ar@{=}[r] & \mathrm{Gal}(N/M) .\\
}\]
Here, $H_1(\partial V_K) \to \mathrm{Gal}(\partial V_{K'}/\partial V_K)$ is surjective by Theorem \ref{local CFT for tori}.
Thus, from the diagram we see $\rho_{N/M} \circ \iota_K(H_1(\partial V_K))= \{1\} \Longleftrightarrow 
\mathrm{Gal}(\partial V_{K'}/\partial V_K)=\{1\} \Longleftrightarrow  K$ is completely decomposed in $N$.

(2) For the proof of the if part, $\mathbb{Z}[m]$ is 0 in $\mathrm{Gal}(\partial V_{K'}/\partial V_K)\cong H_1(\partial V_{K})/h_*(H_1(\partial V_{K'}))$ by the above diagram. 
It implies $h_*(H_1(\partial V_{K'})$ is containing $\mathbb{Z}[m]$. Therefore $h:V_{K'} \to V_K$ is unbranched subcovering by Corollary \ref{cor of local CFT for tori}. 
Thus, $K$ is unbranched in $N$.

For the only if part, by Corollary \ref{cor of local CFT for tori}, there exist a subgroup $H \subset H_1(\partial V_K)$ such that $\mathrm{Gal}(\partial V_{K'}/\partial V_K) \cong H_1(\partial V_{K})/H$ and $H$ containing $\mathbb{Z}[m]$. 
By the above diagram, $\rho_{N/M} \circ \iota_K(H)= \{1\}$. It implies $\rho_{N/M} \circ \iota_K(\mathbb{Z}[m])= \{1\}$.
\end{proof}


H. Niibo \\
Faculty of Mathematics, Kyushu University \\
744, Motooka, Nishi-ku, Fukuoka, 819-0395, Japan \\
E-mail: h.niibo.411@s.kyushu-u.ac.jp


\begin{thebibliography}{Utah}
 \bibitem[Ka]{kapranov} M. Kapranov, Analogies between number fields and 3-manifolds, unpublished note, Max-Plank-Institute.
 \bibitem[KNT]{kato} K. Kato, N. Kurokawa, T. Saito, Number Theory 2: Introduction to Class Field Theory, American Mathematical Soc., 2011.
 \bibitem[Ma]{mazur} B. Mazur, Remark on Alexander polynomial, 1963-64 unpublished notes.
 \bibitem[Mo1]{mor2} M. Morishita, On certain analogies between knots and primes, J. Reine Angew. Math., \textbf{550} (2002), 141-167.
 \bibitem[Mo2]{mor3} M. Morishita, Analogies between knots and primes, 3-manifolds and number rings, SUGAKU EXPOSITIONS, \textbf{23}, No.1, (2010), 1-30.
 \bibitem[Mo3]{mor1} M. Morishita, Knots and primes. An introduction to arithmetic topology. Universitext. Springer, London, 2012.
 \bibitem[N]{Neukirch} J. Neukirch, Algebraic Number Theory, Grundlehren Der Mathematischen Wissenschaften. Volume 322. Springer-Verlag, 1999.
 \bibitem[R1]{reznikov1} A. Reznikov, Three-manifolds class field theory (homology of coverings for a nonvirtually b1-positivemanifold), Selecta Math. (N.S.), \textbf{3}(3)(1997), 361-399.
 \bibitem[R2]{reznikov2} A. Reznikov, Embedded incompressible surfaces and homology of ramified coverings of three-manifolds, Sel. math. New ser. \textbf{6} (2000), 1-39.
 \bibitem[S1]{sikora1} A. Sikora, Idelic Topology, Notes for personal use, unpublished note.
 \bibitem[S2]{sikora2} A. Sikora, Slides from a talk at the workshop "Low dimensional topology and number theory III", March, 2011, Fukuoka.
\end{thebibliography}
\end{document}